\documentclass[12pt]{article}

\usepackage{amssymb,amsmath,amsfonts,amssymb}
\usepackage{graphics,graphicx,color,verbatim}
\usepackage{bbm} 
\usepackage{euscript} 
\usepackage{authblk}
\def\@abssec#1{\vspace{.05in}\footnotesize \parindent .2in
{\bf #1. }\ignorespaces}

\graphicspath{{/EPSF/}{Figures/}}

\newcommand{\be}{\begin{equation}}
\newcommand{\ee}{\end{equation}}
\newcommand{\bea}{\begin{eqnarray}} 
  \newcommand{\eea}{\end{eqnarray}} 
\newcommand{\bee}{\begin{eqnarray*}}
\newcommand{\eee}{\end{eqnarray*}}

\DeclareRobustCommand{\tolaw}[0]{\mathop{\stackrel{(law)}{\longrightarrow}}}

\newtheorem{theorem}{Theorem}
\newtheorem{lemma}{Lemma}[section]
\newtheorem{proposition}[theorem]{Proposition}

\newtheorem{remark}[theorem]{Remark}

\def \Rm {\mathbb R}
\def \Nm {\mathbb N}
\def\un{{\mathbbmss{1}}} 

\def \Cm {\mathbb C}

\def\calL{{\mathcal L}}
\def\calC{{\mathcal C}}

\def\calF{{\mathcal F}}

\newcommand{\eps}{\varepsilon}

\newcommand{\E}{\mathbb E}

\renewcommand{\P}{\mathbb P}

\def\fref#1{{\rm (\ref{#1})}}

\newcommand{\cout}[1]{}

 \renewcommand{\arraystretch}{1.5}

\title{A note on stochastic Schr\"odinger equations with fractional multiplicative noise}
\author{Olivier Pinaud \footnote{pinaud@math.colostate.edu}}
\affil{Department of Mathematics, Colorado State University\\ Fort Collins, CO 80523}

\begin{document}
 
\maketitle

  
 \begin{abstract}
 This work is devoted to non-linear stochastic Schr\"odinger equations with multiplicative fractional noise, where the stochastic integral is defined following the Riemann-Stieljes approach of Z\"ahle. Under the assumptions that the initial condition is in the Sobolev space $H^q(\Rm^n)$ for a dimension $n$ less than three and $q$ an integer greater or equal to zero, that the noise is a $Q-$fractional Brownian motion with Hurst index $H\in(\frac{1}{2},1)$ and spatial regularity $H^{q+4}(\Rm^n)$, as well as appropriate hypotheses on the non-linearity, we obtain the local existence of a unique pathwise solution in $\calC^0(0,T,H^q(\Rm^n)) \cap \calC^{0,\gamma}(0,T,H^{q-2}(\Rm^n))$, for any $\gamma \in [0,H)$. Contrary to the parabolic case, standard fixed point techniques based on the mild formulation of the SPDE cannot be directly used because of the weak smoothing in time properties of the Schr\"odinger semigroup. We follow here a different route and our proof relies on a change of phase that removes the noise and leads to a Schr\"odinger equation with a magnetic potential that is not differentiable in time. 
\end{abstract}
 

\renewcommand{\thefootnote}{\fnsymbol{footnote}}
\renewcommand{\thefootnote}{\arabic{footnote}}

\renewcommand{\thefootnote}{\fnsymbol{footnote}}
\renewcommand{\thefootnote}{\arabic{footnote}}

\renewcommand{\arraystretch}{1.1}






\section{Introduction}
\label{sec:intro}
This work is concerned with the existence theory for the stochastic Schr\"odinger equation with fractional multiplicative random noise of the form
\be \label{SSE}
\left\{
\begin{array}{l}
d \Psi=i \Delta \Psi dt-i \Psi  d B_t^H -ig(\Psi)dt, \qquad t>0, \qquad x \in \Rm^n, \quad n\leq 3\\
\Psi(t=0,\cdot)=\Psi_0,
\end{array}
\right.
\ee
where $B_t^H\equiv B^H(t,x)$ is an infinite dimensional fractional Brownian motion in time and smooth in the space variables. The sense of the stochastic integral will be precised later on and the term $g(\Psi)$ is non-linear. We limit ourselves to $n \leq 3$ for physical considerations, but the theory should hold for arbitrary $n$ with  adjustments of some hypotheses. Our interest for such a problem is motivated by the study of the propagation of paraxial waves in random media that are both strongly oscillating and slowly decorrelating in the variable associated to the distance of propagation. Such media are encountered for instance in turbulent atmosphere or in the earth's crust \cite{dolan,sidi}. More precisely, it is well-known that the wave equation reduces in the paraxial approximation \cite{TAP} to the Schr\"odinger equation on the enveloppe function $\Psi$, which reads in three dimensions
$$
i \partial_z \Psi=-\Delta_\bot \Psi + V(z,x) \Psi, \qquad z \in \Rm^+, \quad x\in \Rm^2,
$$
where $z$ is the direction of propagation of the collimated beam, $x=(x_1,x_2)$ is the transverse plane, $\Delta_\bot=\partial_{x_1^2}^2+\partial_{x_2^2}^2$, and $V$ is a random potential accounting for the fluctuations of the refraction index. If $V$ is stationary and its correlation function $R(z,x):=\E \{V(z+u,x+y) V(u,y)\}$ has the property that
$$
R(z,x) \underset{z \to \infty}{\sim} z^{-\alpha} R_0(x),
$$
where $R_0$ is a smooth function and $0<\alpha<1$, then the process $V$ presents long-range correlations in the $z$ variable since $R$ is not integrable. Rescaling $V$ as $V \to \eps^{-\frac{\alpha}{2}} V(z/\eps,x)$ and invoking the non-central limit theorem, see e.g. \cite{taqqu}, one may expect formally when $\eps \to 0$ that 
$$
\frac{1}{\eps^{\frac{\alpha}{2}}} V(\frac{z}{\eps},x) \Psi^\eps(z,x)  \tolaw \Psi(z,x) dB^H(z,x),
$$
where $B^H$ is a Gaussian process with correlation function
\be \label{corrB}
\E\{B^H(z,x+y)B^H(z',x) \}=\frac{1}{2H(2H-1)} (z^{2H}+(z')^{2H}-|z-z'|^{2H})R_0(x),
\ee
with $H=1-\frac{\alpha}{2} \in (\frac{1}{2},1)$. Proving this fact is an open problem, while the short-range case (when $R$ is integrable) was addressed in \cite{BCF-SIAP-96,Garnier-ITO}, and the limiting wave function $\Psi$ is shown to be a solution to the It\^o-Schr\"odinger equation. Our starting point here is \fref{SSE}, where we added a non-linear term $g(\Psi)$ to account for possible non-linearities arising for instance in non-linear optics.

Let us be more precise now about the nature of the stochastic integral in \fref{SSE}. Since \fref{SSE} is obtained after formal asymptotic limit of a $L^2$ norm preserving Schr\"odinger equation, one may legitimately expect the limiting equation to also preserve the $L^2$ norm. The appropriate stochastic integral should therefore be of Stratonovich type, which in the context of fractional Brownian motions are encountered in the literature as pathwise integrals of various types, e.g. symmetric, forward, or backward, \cite{biagini-book,zahle}. Since in our case of interest the Hurst index $H$ is greater than $\frac{1}{2}$, all integrals are equivalent and can be seen as Riemann-Stieljes integrals of appropriate functions, see \cite{biagini-book,zahle} and section \ref{prelim} for more details. Such an integral is well-defined for instance  if both integrands are of H\"older regularity with respective indices $\beta$ and $\gamma$ such that $\beta+\gamma>1$ \cite{zahle}. 

In the context of SPDEs, the infinite dimensional character of the Gaussian process is usually addressed within two frameworks, whether for standard or fractional Brownian motions: the $Q-$(fractional) Brownian type, or the cylindrical type, see \cite{daprato}. The first class is more restrictive and requires the correlation operator $Q$ in the space variables to be a positive trace class operator (or even more for fractional Brownian motions, see \cite{mas-nualart}); in the second class, it is only supposed that $Q$ is a positive self-adjoint operator on some Hilbert space with appropriate Hilbert-Schmidt embeddings. As was done in \cite{mas-nualart} for parabolic equations with multiplicative fractional noise, we will assume our noise is of $Q-$fractional type, which yields direct pathwise (almost sure) estimates on $B^H$ in some functional spaces. The cylindrical case is more difficult and our approach does not seem to generalize to it. The $Q-$fractional case actually excludes stationary in $x$ correlation functions of the form \fref{corrB} since they lead to a cylindrical type noise, which is a drawback of our assumptions. This latter case, even in the more favorable situation of parabolic equations, seems to still be open. Standard Brownian motions are more amenable to cylindrical noises since the It\^o isometry holds. In the case of fractional type integrals,  the ``It\^o isometry'' involves the Malliavin derivative of the process, which is difficult to handle in the context of SPDEs with multiplicative noise. Hence, an existence theory for the Schr\"odinger equation in some average sense seems more involved to achieve, and we thus focus on a pathwise theory which requires the $Q-$Brownian assumption in our setting. 

Stochastic ODEs with fractional Brownian motion were investigated in great generality in \cite{nualart-rasc}. Stochastic PDEs with fractional multiplicative noise are somewhat difficult to study and to the best of our knowledge, the most advanced results in the field are that of Maslowski and al \cite{mas-nualart}, Duncan et al \cite{duncan-mas-mult} or Grecksch et al \cite{grecksch}. The reference \cite{duncan-mas-mult} involves finite dimensional fractional noises, which is a limitation. Several other works deal with additive noise, which is a much more tractable situation as stochastic integrals are seen as Wiener integrals \cite{duncan-mas,tindel,gautier} and cylindrical type noises are allowed. References \cite{mas-nualart,duncan-mas-mult,grecksch} consider variations of parabolic equations of the form
\be \label{equ}
d u= A u dt+u d B_t^H,
\ee
where $A$ is the generator of an analytic semigroup $S$, and the equation can be complemented with non-linear terms and a time-dependency in $A$ in \cite{duncan-mas-mult}. The noise $B^H$ in these references is a $Q-$fractional Brownian with possibly additional assumptions. The difficulty is naturally to make sense of the term $u d B_t^H$ and to show that $u$ is H\"older in time. In that respect, the analyticity hypothesis is crucial: indeed, the standard technique to analyze \fref{equ} is to use mild solutions of the form
$$
u(t)=S(t)u_0+\int_0^t S(t-s)u(s)d B^H_s,
$$
and for the integral to exist, one needs the term $S(t-s)u(s)$ to be roughly of H\"older regularity in time with index greater than $1-H$. This means that both $u$ and the semigroup $S$ need such a regularity. While the term $u(s)$ can be treated in the fixed point procedure, the semigroup $S(t-s)$ has to be sufficiently smooth in time, which holds for analytic semigroups, but not in the case of $C_0$ unitary groups generated by $i \Delta$ in the Schr\"odinger equation. In the latter situation, one can ``trade'' some regularity in time for $S$ with some spatial regularity on $u$, but this procedure does not seem to be exploitable in a fixed point procedure. Another possibility could be to take advantage of the regularizing properties of the Schr\"odinger semigroup that provide a gain of almost half a derivative in space, and therefore to almost a quarter of a derivative in time \cite{cazenave}. It looked to us rather delicate to follow such an approach since the smoothing effects hold for particular topologies involving spatial weights which looked fairly intricate to handle in our problem, even by using the classical exchange regularity/decay for the Schr\"odinger equation. The strictly linear case (i.e. when $g=0$ in \fref{SSE}) can likely be treated by somewhat brute force with iterated Wiener integrals and the Hu-Meyer formula, but this approach does not carry on to the non-linear setting.

We propose in this work a different route than the mild formulation and a quite simple remedy based on two direct observations: (i) the usual change of variables formula holds for the pathwise stochastic integral and (ii) using it along with a change of phase removes the noise and leads to a Schr\"odinger equation with magnetic vector potential $A(t,x)=-\nabla B^H(t,x)$. Forgetting for the moment the non-linear term $g(\Psi)$, introducing $\varphi(t,x):=e^{i B^H(t,x)} \Psi(t,x)$ the filtered wavefunction, and supposing without lack of generality that $B^H(0,x)=0$, this yields the system
\be \label{SEM}
\left\{
\begin{array}{l}
i \partial_t \varphi = -\Delta_{B_t^H} \varphi, \qquad \Delta_{B_t^H} =e^{i B_t^H}\circ \Delta \circ e^{-i B_t^H}\\
\varphi(t=0,\cdot)=\Psi_0,
\end{array}
\right.
\ee
which is a standard Schr\"odinger equation with a time-dependent Hamiltonian. There is a vast literature on the subject, see \cite{RS-80-2, debouard_mag,kato0, katolinear,nakamura, michel,yajima,yajima2,zagrebnov} for a non-exhaustive list. One of the most classical assumptions on $A$ for the existence of an evolution operator generated by the Hamiltonian $\Delta_{B_t^H}$ is that $A$ is a $\calC^1$ function in time with values in $H^1(\Rm^n)$. This is of course not verified for the fractional Brownian motion. The price to pay for that is to require additional spatial regularity, and one possibility (likely not optimal) is to suppose that $B^H$  has values in $H^4(\Rm^n)$. Assuming such a strong regularity is naturally a  drawback in this approach.

Regarding the treatment of the non-linearity, we suppose that it is invariant by a change of phase, that is $e^{iB^H_t} g(\Psi)=g(\varphi)$, which is verified by power non-linearities of the form $g(\Psi)=|\Psi^\sigma| \Psi$ or by $g(\Psi)=V[\Psi] \Psi$ where $V$ is the Poisson potential. Contrary to the case of non-linear It\^o-Schr\"odinger equations where various tools such as Strichartz estimates or Morawetz estimates have been successfully used to investigate focusing/defocusing phenomena in random and deterministic settings \cite{cazenave,debu2,debu3,debu4}, there are very few available techniques  to study \fref{SEM} with a potential vector $A$ not smooth in time and augmented with the term $g$. There are Strichartz estimates in the context of magnetic Schr\"odinger equations, but some require $A$ to be $\calC^1$ in time \cite{yajima}, and some others avoid such an hypothesis but assume instead that $A$ is small in some sense \cite{stefanov_mag}, which has no reason to hold here. As a result, we are lead to make rather crude assumptions on $g$ in order to obtain a local existence result. Moreover, the analysis of non-linear Schr\"odinger equations generally relies in a crucial manner on energy methods. In our problem of interest, we are only able to obtain energy conservation for smooth solutions, which turns out to be of no use when trying to obtain a global-in-time result and limits us to local results, unless the non-linearity is globally Lipschitz in the appropriate topology. This is due to the fractional noise that does not allow us to obtain $H^1$ estimates for $\Psi$ via the energy relation, as we explain further in remark \ref{ener}.

The main result of the paper is therefore a local existence result of pathwise solutions to \fref{SSE} with a smooth $Q-$fractional noise $B^H_t$ and appropriate assumptions on the non-linearity $g$. The  article is structured as follows: in section \ref{prelim}, we recall basic results on fractional stochastic integration, and present our main result in section \ref{main}; section \ref{secmag} is devoted to the magnetic Schr\"odinger equation \fref{SEM}, while section \ref{back} concerns the proof of our main theorem.

\section{Preliminaries} \label{prelim}
\noindent \textbf{Notation.} We denote by $H^k(\Rm^n)$ and $W^{k,q}(\Rm^n)$, $1 \leq n \leq 3$, the standard Sobolev spaces with the convention that $H^0(\Rm^n):=L^2(\Rm^n)$. For a Banach space $V$, $T>0$, and $0<\alpha<1$, $W_{\alpha,1}(0,T,V)$ 
denote the space or mesurable functions $f: [0,T] \to V$  equipped with the norm
$$
\|f \|_{\alpha,1,V}=\int_0^T\left(\frac{\|f(s)\|_V}{s^\alpha}+\int_0^s\frac{\|f(s) -f(\tau)\|_V}{(s-\tau)^{\alpha+1}} d\tau \right) ds.
$$
The space $\calC^{0,\alpha}(0,T,V)$ denotes the classical H\"older space of functions with values in $V$. When $V=\Cm$ or $\Rm$, we will simply use the notations $W_{\alpha,1}(0,T)$, $\calC^{0,\alpha}(0,T)$ and $\|\cdot \|_{\alpha,1}$. Notice that for any $\eps>0$, $\calC^{0,\alpha+\eps}(0,T,V) \subset W_{\alpha,1}(0,T,V)$. For two Banach spaces $U$ and $V$, $\calL(U,V)$ denotes the space of bounded operators from $U$ to $V$, with the convention $\calL(U)=\calL(U,U)$. The $L^2$ inner product is denoted by $(f, g)=\int_{\Rm^n} \overline{f}g dx$ where $\overline{f}$ is the complex conjugate of $f$.\\

\noindent \textbf{Fractional Brownian motion.} For some positive time $T$, we denote by $\beta^H=\{ \beta^H(t), \; t \in [0,T]\}$ a standard fractional Brownian motion (fBm) over a probability space  $(\Omega, \calF,\P)$ with Hurst index $H \in (\frac{1}{2},1)$. We will denote by $L^2(\Omega)$ the space of square integrable random variables for the measure $\P$ and will often omit the dependence of $\beta^H$ on $\omega \in \Omega$ for simplicity. The process $\beta^H$ is a centered Gaussian process with covariance 
$$
\E\{ \beta^H_t\beta^H_s \}=\frac{1}{2} (t^{2H}+s^{2H}-|t-s|^{2H}).
$$
Since $\E\{ (\beta^H_t-\beta^H_s)^2\}=|t-s|^{2H}$, $\beta^H$ admits a H\"older continuous version with index strictly less than $H$. In order to definite the infinite dimensional noise $B^H(t,x)$, consider a sequence of independent fBm $(\beta_n^H)_{n \in \Nm}$. Let $Q$ be a positive trace class operator on $L^2(\Rm^n)$ and denote by $(\mu_n, e_n)_{n \in \Nm}$ its spectral elements. For $V=H^{q+4}(\Rm^n)$, $q$ non-negative integer, and $\lambda_n=\sqrt{\mu_n}$, we assume that
\be \label{assumQ}
\sum_{p \in \Nm} \lambda_p \|e_p\|_V < \infty.
\ee
The process $B^H(t,x)$ is then formally defined by
$$
B^H(t,x):=\sqrt{Q} \sum_{p \in \Nm} e_p(x) \beta^H_p(t)=\sum_{p \in \Nm} \lambda_p  e_p(x) \beta^H_p(t).
$$
The sum is normally convergent in $\calC^{0,\gamma}(0,T,V)$, $\P$ almost surely for $0\leq \gamma<H$. Indeed, in the same fashion as \cite{mas-nualart}, let 
$$
K(\omega)=\sum_{p \in \Nm} \lambda_p \|e_p\|_V \|\beta^H_p(\cdot, \omega)\|_{C^{0,\gamma}(0,T)}
$$ so that by monotone convergence
$$
\E K=\sum_{p \in \Nm} \lambda_p \|e_p\|_V \E \|\beta^H_p\|_{C^{0,\gamma}(0,T)}.
$$
According to \cite{nualart-rasc} Lemma 7.4, for every $T>0$ and $\eps>0$, there exists a positive random variable $\eta_{\eps,T,p}$ where $\E\{|\eta_{\eps,T,p}|^q \}$ is finite for $1 \leq q < \infty$ and independent of $p$ since the $\beta_p^H$ are identically distributed, such that $|\beta^H_p(t)-\beta^H_p(s)| \leq \eta_{\eps,T,p} |t-s|^{H-\eps}$ almost surely. Hence, thanks to \fref{assumQ} and picking $\gamma=H-\eps$, we have $\E K<\infty$,
\be \label{defK}
K(\omega) < \infty, \quad \P \quad \textrm{almost surely},
\ee
and $B^H$ defines almost surely an element of $\calC^{0,\gamma}(0,T,V)$. As a contrast, a cylindrical fractional Brownian motion is defined for a positive self-adjoint $Q$, which does not provide us with almost sure bounds on $B^H$ in  $\calC^{0,\gamma}(0,T,V)$. Suppose indeed that $\sqrt{Q}$ is a convolution operator of the form $\sqrt{Q} u= g*u$ for some smooth real-valued kernel $q$ and that $(e_p)_{p \in \Nm}$ is a  real-valued basis of $L^2(\Rm^n)$. Then, the resulting correlation function is stationary (this follows from the convolution and is motivated by \fref{corrB}) and 
$$
\E \{(B^H(t,x)-B^H(s,x))^2\}= |t-s|^{2H}\sum_{p \in \Nm} (g*e_p(x))^2=|t-s|^{2H} \|g\|^2_{L^2}
$$
so that $B^H$ belongs to $\calC^{0,\gamma}(0,T,L^\infty(\Rm^n, L^2(\Omega)))$ for $0\leq \gamma \leq H$. As explained in the introduction, we are not able to handle such a noise since integration in the probability space is required beforehand in order to get some estimates. This is not an issue in the context of standard Brownian motions or additive fractional noise, but leads to technical difficulties here.\\

\noindent \textbf{Fractional stochastic integration.} We follow the approach of \cite{mas-nualart,nualart-rasc} based on the work of Z\"ahle \cite{zahle} and introduce the so-called Weyl derivatives defined by, for any $\alpha \in (0,1)$ and $t \in (0,T)$:
\bee
D_{0+}^\alpha f(t)&=&\frac{1}{\Gamma(1-\alpha)} \left(\frac{f(t)}{t^\alpha}+\alpha \int_0^t \frac{f(t) -f(s)}{(t-s)^{\alpha+1}} ds \right)\\
D_{T-}^\alpha f(t)&=&\frac{(-1)^\alpha}{\Gamma(1-\alpha)} \left(\frac{f(t)}{(T-t)^\alpha}+\alpha \int_t^T \frac{f(t) -f(s)}{(s-t)^{\alpha+1}} ds \right),
\eee
whenever these quantities are finite. Above, $\Gamma$ stands for the Euler function. Following \cite{zahle}, the generalized Stieljes integral of a function $f  \in  \calC^{0,\lambda}(0,T)$ against a function $g \in \calC^{0,\gamma}(0,T)$ with $\lambda+\gamma>1$, $\lambda>\alpha$ and $\gamma>1-\alpha$ is defined by
\be \label{stiel}
\int_0^T f dg:=(-1)^\alpha \int_0^T D_{0+}^\alpha f(s) D_{T-}^{1-\alpha} g_{T-}(s) ds,
\ee
with $g_{T-}(s)=g(s)-g(T-)$. The definition does not depend on $\alpha$ and
$$
\int_0^t f dg:=\int_0^T f \un_{(0,t)} dg.
$$
The integral can be extended to different classes of functions since, see \cite{nualart-rasc},
\be \label{estimint}
\left|\int_0^T f dg \right| \leq \|f \|_{\alpha,1} \Lambda_\alpha(g),
\ee
where
$$
\Lambda_\alpha(g):=\frac{1}{\Gamma(1-\alpha) \Gamma(\alpha)} \sup_{0<s<t<T} \left(\frac{|g(t)-g(s)|}{(t-s)^{1-\alpha}}+\alpha \int_s^t \frac{|g(\tau) -g(s)|}{(\tau-s)^{2-\alpha}} d\tau\right),
$$
so that the integral is well-defined if $f \in W_{\alpha,1}(0,T)$ and $\Lambda_\alpha(g)<\infty$. Besides, the fractional integral satisfies the following change of variables formula, see \cite{zahle}: let $F \in \calC^1(\Rm \times [0,T])$, $g \in \calC^{0,\lambda}(0,T)$ and  $\partial_1 F(g(\cdot),\cdot) \in \calC^{0,\gamma}(0,T)$ with $\lambda+\gamma>1$, then
\be \label{chain}
F(g(t),t)-F(g(s),s)=\int_s^t \partial_2 F(g(\tau),\tau) d\tau+
\int_s^t \partial_1 F(g(\tau),\tau) d g(\tau),
\ee
where $\partial_j F$, $j=1,2$ denotes the partial derivative of $F$ with respect to the $j$ coordinate.

For some Banach space $U$ and an operator-valued random function $F \in W_{\alpha,1}(0,T,\calL(V,U))$ almost surely for some $\alpha \in (1-H,\frac{1}{2})$, the stochastic integral of $F$ with respect to $B^H$ is then formally defined by
\be \label{defintsto}
\int_0^t F_s d B_s^H:=\sum_{p \in \Nm} \lambda_p \int_0^t F_s(e_p) d\beta_p^H(s).
\ee
The integral  defines almost surely an element of $U$ for all $t\in[0,T]$ since by Jensen's inequality for the second line
\bee
\sum_{p \in \Nm} \lambda_p \left\| \int_0^t F_s(e_p) d\beta_p^H(s) \right\|_U &\leq& \sum_{p \in \Nm} \lambda_p  \Lambda_\alpha(\beta^H_p)  \left\| \|F_s(e_p) \|_{\alpha,1} \right\|_U \\
&\leq& C \|F_s\|_{\alpha,1,\calL(U,V)} \sum_{p \in \Nm} \lambda_p \|e_p\|_V \Lambda_\alpha(\beta_p^H)
\eee
and as shown in \cite{mas-nualart}, 
\be \label{finitelamb}
\sum_{p \in \Nm} \lambda_p \|e_p\|_V \Lambda_\alpha(\beta_p^H(\cdot,\omega)) <\infty \qquad \P \quad \textrm{almost surely}.
\ee
Hence \fref{defintsto} is well-defined and the convergence of the sum has to be understood as the $\P$ almost sure convergence in $U$.

We will use the following two results: the first Lemma is a generalization of the change of variables formula \fref{chain} to the infinite dimensional setting, and the second a version a the Fubini theorem adapted to the stochastic integral. Their proofs are given in the appendix. Below, $V=H^{q+4}(\Rm^n)$.
\begin{lemma} \label{chain2} Let $F: V \times [0,T] \to \Cm$ be a continuously differentiable function. Let $\partial_1 F$ be the differential of $F$ with respect to the first argument and $\partial_2 F$ be its partial derivative with respect to the second. For every $v \in V$ and $B \in \calC^{0,\gamma}(0,T,V)$ for any $0\leq 
\gamma<H$, let $\phi(t):=\partial_1 F(B_t,t)(v)$. Assume that $\phi \in \calC^{0,\lambda}(0,T)$ with $\lambda+\gamma>1$, and that there exists a constant $C_M>0$ such that, for all $B$ with $\|B\|_{\calC^{0,\gamma}(0,T,V)} \leq M$:
\be \label{hypphi}
\|\phi\|_{\calC^{0,\lambda}(0,T)} \leq C_M \|v\|_V.
\ee
Then, we have the change of variables formula, $\forall (s,t) \in [0,T]^2$, $\P$ almost surely:
$$
F(B_t^H,t)-F(B_s^H,s)=\int_s^t \partial_2 F(B_{\tau}^H,\tau) d\tau+
\sum_{p \in \Nm} \lambda_p \int_s^t \partial_1 F(B_{\tau}^H,\tau)(e_p) d \beta^H_p(\tau).
$$
\end{lemma}

\begin{lemma} \label{fubini}Let $F \in W_{\alpha,1}(0,T,\calL(V,L^1(\Rm^n)))$ with $1-H<\alpha<\frac{1}{2}$. Then we have: 
$$
\sum_{p \in \Nm} \lambda_p \int_s^t \left(\int_{\Rm^n}  F_{\tau,x}(e_p) dx \right) d \beta^H_p(\tau) =\int_{\Rm^n} \left(\int_s^t F_{\tau,x} dB_\tau^H \right)dx.
$$
\end{lemma}

\section{Main result} \label{main}

We present in this section the main result of the paper. We precise first in which sense \fref{SSE} is understood. We say that $\Psi \in \calC^0(0,T,H^q(\Rm^n)) \cap \calC^{0,\gamma}(0,T,H^{q-2}(\Rm^n))$, for all $0\leq \gamma<H$, $q$ non-negative integer, is a solution to \fref{SSE} if it verifies for all test function $w\in \calC^1(0,T,H^{q+2}(\Rm^n))$, for all $t \in [0,T]$ and $\P$ almost surely
\begin{align} \label{defsol} \nonumber
&\left( \Psi(t), w(t)   \right)-\left( \Psi_0, w(0) \right)
=\int_0^t  \left( \Psi(s), \partial_s w(s) \right) ds
\\[3mm]
& -i\int_0^t \left( \Psi(s),\Delta w(s)\right) ds  +i  \int_0^t\left(\Psi(s), w(s)  d B_s^H\right)+i\int_0^t \left(g(\Psi(s)),w(s) \right) ds, 
\end{align}
where the term involving the stochastic integral is understood as
$$
\int_0^t\left(\Psi(s), w(s)  d B_s^H \right):=\sum_{p \in \Nm} \lambda_p \int_0^t\left(\Psi(s)e_p, w(s)\right) d \beta^H_p(s).
$$
The latter is well-defined since the mapping $F_s: e_p \mapsto (\Psi e_p, w)$ belongs to $\in \calC^{0,\gamma}(0,T,\calL(V,\Rm))$ thanks to standard Sobolev embeddings for $n \leq 3$. We assume the following hypotheses on the non-linear term $g$:\\

\textbf{H}: We have $g( e^{i \theta(t,x)} \Psi)= e^{i \theta(t,x)} g(\Psi)$ for all real function $\theta$, and for any $\Psi_1, \Psi_2$ in $H^q(\Rm^n)$ with $\|\Psi_i\|_{H^s}\leq M$, $i=1,2$, there exist $p \in \{0,\cdots,q\}$ and positive constants $C_M$ and $C'_M$ such that
\bee
\| g(\Psi_1) \|_{H^q} &\leq& C_M \|\Psi\|_{H^q}\\
\| g(\Psi_1)-g(\Psi_2) \|_{H^p} &\leq& C'_M \|\Psi_1-\Psi_2\|_{H^p}.
\eee

The main result of this paper is the following:
\begin{theorem} \label{th1} Assume that \textbf{H} is satisfied. Suppose moreover that \fref{assumQ} is verified for $V=H^{q+4}(\Rm^n)$, $q$ non-negative integer. Then, for every $\Psi_0 \in H^q(\Rm^n)$, there exists a maximal existence time $T_M>0$ and a unique function $\Psi \in \calC^0(0,T_M,H^{q}(\Rm^n)) \cap\calC^{0,\gamma}(0,T_M,H^{q-2}(\Rm^n))$, $0\leq \gamma <H$, verifying \fref{defsol} for all $t\in[0,T_M]$ $\P$ almost surely. Moreover, $\Psi$ admits the following representation formula:
\be \label{repre_th}
\Psi(t)= e^{-i B^H_t} U(t,0) \Psi_0+ e^{-i B^H_t} \int_0^t U(t,s) e^{i B^H_s} g(\Psi(s))ds,
\ee
where $U=\{U(t,s)\}$ is the evolution operator generated by the operator
$$
i\Delta_{B_t^H} =i e^{i B_t^H}\circ \Delta \circ e^{-i B_t^H}.
$$
If in addition $\Im g(\Psi) \overline{\Psi}=0$, then for all $t\in[0,T_M]$ the charge conservation holds:
$$
\| \Psi(t)\|_{L^2}=\| \Psi(0)\|_{L^2}.
$$
If $g$ is globally Lipschitz in $H^q(\Rm^n)$, then the solution exists for all time $T<\infty$. 
\end{theorem}

When $d=3$, a classical example of a non-linearity satisfying \textbf{H} for $q=p=1$ is $g(\Psi)=V[\Psi] \Psi$, where $V[\Psi]$ is the Poisson potential defined by
$$
V[\Psi](x)=\int_{\Rm^3} \frac{|\Psi(y)|^2}{|x-y|}dy.
$$
Indeed, $g$ is locally Lipschitz in $H^1(\Rm^3)$: let $\Psi_1, \Psi_2 \in H^1(\Rm^n)$; thanks to the Hardy-Littlewood-Sobolev inequality \cite{RS-80-2}, Chapter IX.4, as well as standard Sobolev embeddings, we have
$$
\|\nabla V[\Psi_1] -\nabla V_2[\Psi_2] \|_{L^3} \leq C \| |\Psi_1|^2-|\Psi_2|^2 \|_{L^{\frac{3}{2}}} \leq C \| \Psi_1 -\Psi_2 \|_{L^2}\| \Psi_1 +\Psi_2 \|_{H^1}
$$
and direct computations yield
$$
\| V [\Psi_1] \|_{L^\infty} \leq C \|\Psi_1\|^2_{L^2}+C\||\Psi_1|^2\|_{L^2}.
$$
Hence,
\begin{align*}
&\| g(\Psi_1)-g(\Psi_2)\|_{H^1} \\
&\qquad \leq  C\|V[\Psi_1] \|_{L^\infty} \|\Psi_1 -\Psi_2\|_{L^2}+ C\|\Psi_1-\Psi_2 \|_{H^1}\|\Psi_1+\Psi_2 \|_{H^1}\|\Psi_2\|_{L^2}\\
&\qquad \qquad +C\|V[\Psi_1] \|_{L^\infty} \|\nabla \Psi_1 - \nabla \Psi_2\|_{L^2}+C\| \Psi_1 -\Psi_2 \|_{L^2}\| \Psi_1 +\Psi_2 \|_{H^1} \|\Psi_2 \|_{L^6}\\
&\qquad \leq C( \| \Psi_1\|^2_{H^1}+ \| \Psi_2\|^2_{H^1})  \|\Psi_1-\Psi_2 \|_{H^1}.
\end{align*}


Another example is given by power non-linearities of the form $g(\Psi)=\mu |\Psi|^{2\sigma} \Psi$ for some $\mu \in \Rm$ and $\sigma>0$. A $L^\infty$ bound is needed on $\Psi$ for \textbf{H} to be verified. When $n>1$, we set then $q=2$ and obtain, for all $\sigma\geq \frac{1}{2}$:
$$
\| g(\Psi) \|_{H^2} \leq C \| \Psi \|^{2\sigma+1}_{H^2}, \qquad \| g(\Psi_1)-g(\Psi_2) \|_{L^2} \leq C \| \Psi_2 +\Psi_1\|^{2\sigma}_{H^2}\| \Psi_1-\Psi_2 \|_{L^2},
$$
while it can be easily shown that $\textbf{H}$ is verified for $n=1$ and $q=1$ for all $\sigma \geq 0$.

\begin{remark} \label{ener} In order to both lower the spatial regularity assumptions on $B^H$, $\Psi_0$, $g$ and to obtain global-in-time results, it is natural to consider the energy conservation identity (derived formally by multiplying \fref{mag} by $\overline{\partial_t \varphi}$ and integrating, and can be justified for classical solutions when $q \geq 2$ using the regularity of $\varphi$ of Theorem \ref{th_mag} and Lemma \ref{chain2}) that reads for $g=0$ for simplicity:
$$
\frac{1}{2} \| \nabla \Psi(t) \|^2_{L^2}=\frac{1}{2} \| \nabla \Psi_0 \|^2_{L^2}-\Im \int_0^t \int_{\Rm^n} \overline{\Psi(s)} \nabla \Psi(s) \cdot \nabla d B^H_s dx.
$$
Unfortunately, it is not clear to us how this identity can be used in order to obtain estimates on $\| \nabla \Psi\|_{W_{\alpha,1}(0,T,L^2)}$ for  $1-H<\alpha<\frac{1}{2}$ that would depend only on $\| \nabla \Psi_0 \|_{L^2}$ and $\| B^H\|_{\calC^{0,\gamma}(0,T,W^{1.\infty})}$, $\frac{1}{2}<\gamma<H$. Indeed, following the lines of the stochastic ODE case of \cite{nualart-rasc} in order to treat the stochastic integral and use the Gronwall Lemma,  what can be deduced from the above relation is an estimate of the form
$$
\left\| \| \nabla \Psi(t,\cdot) \|^2_{L^2} \right\|_{W_{\alpha,1}(0,T)} \leq C+C \int_0^T f(s) \| \nabla \Psi(s,\cdot) \|^2_{_{W_{\alpha,1}(0,T,L^2)}}ds
$$ 
for some positive integrable function $f$ and where the constant $C$ depends on $\| \nabla \Psi_0 \|_{L^2}$ and $\| B^H\|_{\calC^{0,\gamma}(0,T,W^{1.\infty})}$. 
This does not yield the desired bound since we cannot control the term $\| \nabla \Psi(s,\cdot) \|^2_{_{W_{\alpha,1}(0,T,L^2)}}$ by $\left\| \| \nabla \Psi(s,\cdot) \|^2_{L^2} \right\|_{W_{\alpha,1}(0,T)}$. Hence, as opposed to the standard Brownian case, energy methods do not provide us here with an $H^1$ global-in-time estimate. 
\end{remark}
\begin{remark} \label{rem2} When $q \geq 2$, then $\Psi$ is a classical solution to \fref{SSE} in the sense that it satisfies for all $t \in [0,T_m]$, $\P$ a.s., $x$ a.e.:
$$
\Psi(t)=\Psi(0)+i \int_0^t \Delta \Psi(s) ds-i \int_0^t \Psi(s) dB_s^H-i\int_0^t g(\Psi(s)) ds.
$$
A proof of this result is given in the appendix.
\end{remark}

The rest of the paper is devoted to the proof of Theorem \ref{th1}. The starting point is to define $\varphi(t,x)=e^{i B^H(t,x)} \Psi(t,x)$, to use the invariance of $g$ with respect to a change of phase and to formally apply Lemma \ref{chain2} to arrive at 
\be \label{MSE2}
i\partial_t \varphi = -\Delta_{B_t^H} \varphi+g(\varphi).
\ee
Remark that $\Delta_{B_t^H}$ can formally be recast as
$$
\Delta_{B_t^H} =\Delta -2i \nabla B_t^H \cdot \nabla  -|\nabla B_t^H|^2  -i \Delta B_t^H .
$$
In section \ref{secmag}, we construct the evolution operator $U=\{U(t,s)\}$ generated by $i\Delta_{B_t^H}$ and obtain the existence of a unique solution to the latter magnetic Schr\"odinger equation. In section \ref{back}, we use the regularity properties of the function $\varphi$ together with Lemma \ref{chain2} to prove that  $\Psi=e^{-i B^H_t}\varphi$ is the unique solution to \fref{SSE}. The existence follows from showing that $e^{-i B^H_t}\varphi$ is a solution to \fref{defsol}. The uniqueness stems from a reverse argument: owing a solution $\Psi$ to $\fref{defsol}$ with the corresponding regularity, we show that $\Psi e^{i B^H_t}$ is a solution to \fref{MSE2}. This requires some regularization since the function $ e^{-i B^H_t} z$ for $z$ smooth cannot be used as a test function in \fref{defsol}, as well as the interpretation of a classical integral involving a full derivative as a fractional integral.

\section{Existence theory of the magnetic Schr\"odinger equation} \label{secmag}
The first part of this section consists in constructing the evolution operator $U$. We follow the classical methods of Kato \cite{katolinear} and \cite{pazy}. The second part is devoted to the existence theory for the linear magnetic Schr\"odinger equation, which is then used for the non-linear case.

\subsection{Construction of the evolution operator}
We follow here the construction of \cite{pazy}, Chapter 5. Let $X$ and $Y$ be Banach spaces with norms $\|\cdot \|$ and $\| \cdot \|_Y$, where $Y$ is densely and continuously embedded in $X$. For $t \in [0,T]$, let $A(t)$ be the infinitesimal generator of a $C_0$ semigroup on $X$. Consider the following hypotheses:
\begin{itemize}
\item[(H1)] $\{ A(t)\}_{t\in [0,T]}$ is such that there are constants $\omega_0$ and $M\geq 1$, where $]\omega_0,\infty[ \subset \rho(A(t))$ for $t \in [0,T]$, $\rho(A(t))$ denoting the resolvent set of $A(t)$, and 
$$
\left\|\prod_{j=1}^k e^{- s_j A(t_j )}\right \| \leq M e^{\omega_0 \sum_{j=1}^k s_j }, \qquad s_j \geq 0, \qquad 0 \leq t_1 \leq t_2 \leq  \dots \leq T.
$$
\item[(H2)] There is a family $\{ Q(t)\}_{t\in [0,T]}$ of isomorphisms of $Y$ onto $X$ such that for every $y\in Y$, $Q(t)v$ is continously differentiable in $X$ on $[0,T]$ and
$$
Q(t) A(t) Q(t)^{-1}=A(t)+C(t)
$$
where $C(t)$, $0 \leq t \leq T$, is a strongly continuous family of bounded operators on $X$.
\item[(H3)] For $t\in [0,T]$, $Y \subset D(A(t))$, $A(t)$ is a bounded operator from $Y$ into $X$ and $t \to A(t)$ is continuous in the $\calL(Y,X)$ norm.
\end{itemize}
We then have the following result, see \cite{katolinear}, or \cite{pazy}, Chapter 5, Theorems 2.2 and 4.6:
\begin{theorem}\label{kato}
Assume that (H-1)-(H-2)-(H-3) are verified. Then, there exists a unique evolution operator $U=\{U(t,s)\}$, defined on the triangle $\Delta_T:T\geq t\geq s\geq 0$ such that
\begin{itemize}
\item[(a)] $U$ is strongly continuous on $\Delta_T$ to $\calL(X)$, with $U(s,s)=I$,
\item[(b)] $U(t,r)U(r,s)=U(t,s)$,
\item[(c)] $U(t,s) Y \subset Y$, and $U$ is strongly continuous on $\Delta_T$ to $\calL(Y)$,
\item[(d)] $dU(t,s)/dt=-A(t)U(t,s)$,  $dU(t,s)/ds=U(t,s)A(s)$, which exist in the strong sense in $\calL(Y,X)$, and are strongly continuous $\Delta_T$ to $\calL(Y,X)$.
\end{itemize}
\end{theorem}
In the next result, we show that for suitable functions $B$, the operator $i \Delta_B=i e^{iB} \circ \Delta \circ e^{-iB}$ generates an evolution operator $U$.
\begin{proposition} \label{geneevol} Let $X=L^2(\Rm^n)$ and $Y=H^{2k}(\Rm^n)$, $k \geq 1$, and let $B \in \calC^0(0,T,H^{2k+2}(\Rm^n))$. Then, the operator $i \Delta_B$ generates an evolution operator $U$ satisfying Theorem \ref{kato} and $U$ is an isometry on $L^2(\Rm^n)$.
\end{proposition}
\begin{proof} We verify hypotheses (H-1)-(H-2)-(H-3) for $A(t)=i \Delta_B$. Let $\Delta_{B_t} := \Delta +L(t)$ with 
\be \label{defL}
L(t)=-2i \nabla B_t \cdot \nabla  -|\nabla B_t|^2  -i \Delta B_t .\ee
 First, for $t$ fixed in $[0,T]$, the Kato-Rellich theorem \cite{RS-80-2} yields that $\Delta_{B_t}$ is self-adjoint on $D(\Delta)=H^2(\Rm^n)$.  Indeed, using the regularity $B \in \calC^0([0,T],H^4(\Rm^n))$, it is straightforward to verify that $L(t)$ is symmetric and $\Delta$-bounded with relative bound strictly less than one. We also obtain that $D(\Delta_{B_t})=H^2(\Rm^n)$, $\forall t \in [0,T]$. Stone's theorem \cite{RS-80-I} then implies that for $t$ fixed, $i \Delta_{B_t}$ is the generator of a  $C_0$ unitary group on $X$. Moreover, $-\Delta_{B_t}$ is positive, so that the spectrum of $i \Delta_{B_t}$ lies in $i[0,\infty)$. We therefore conclude that the family $\{i \Delta_{B_t} \}_{t\in [0,T]}$ satisfies hypothesis (H-1).

Regarding (H-2), let $Q=\Delta_{(k)}+I$, where $I$ is the identity operator and 
$$\Delta_{(k)}=(-1)^k \sum_{j=1}^n \partial^{2k}_{x_j^{2k}}, \qquad k \geq 1.$$The operator $Q$ is a positive definite self-adjoint operator on $H^{2k}(\Rm^n)$, and an isomorphism from $H^{2k}(\Rm^n)$ to $L^2(\Rm^2)$. It is also obviously continuously differentiable since it does not depend on $t$. Moreover,
$$
Q \Delta_{B_t} Q^{-1}=\Delta_{B_t}+[Q,\Delta_{B_t}]Q^{-1},
$$
where $[A,B]$ denotes the commutator between two operators $A$ and $B$. We have the following Lemma:
\begin{lemma} For $k \geq 1$, let $B \in \calC^0([0,T],H^{2k+2}(\Rm^n))$. Then $[Q,\Delta_{B_t}]Q^{-1} \in \calL(L^2(\Rm^n))$. 
\end{lemma}

\begin{proof}
We have $[Q,\Delta_{B_t}]Q^{-1}=[\Delta_{(k)},L(t)]Q^{-1}$, and using the product rule
\begin{align*}
&[\Delta_{(k)},L(t)]=\\
&(-1)^k \sum_{j=1}^n \sum_{p=0}^{2k-1} 
\left(
\begin{array}{c}
2k \\p
\end{array}
\right)
\left(-2i \{\nabla \partial^{2k-p}_{x^{2k-p}_j} B \}\cdot \nabla \partial^p_{x^p_j} - \{\partial^{2k-p}_{x^{2k-p}_j}|\nabla B|^2 \}\partial^p_{x^p_j}-i\{ \Delta \partial^{2k-p}_{x^{2k-p}_j} B  \}\partial^p_{x^p_j} \right)
\end{align*}
where $\tiny{\left(
\begin{array}{c}
2k \\p
\end{array}
\right)}$ is the binomial coefficient and there are as usual no terms corresponding to $p=2k$ because of the commutator. Using Standard Sobolev embeddings for $H^{2k+2}(\Rm^n)$ when $n \leq 3$, we have for $j=1,\dots,n$ that $\partial^{2k+1}_{x^{2k+1}_j} B \in L^\infty_t L^p_x$ for $p=6$, $p<\infty$ and $p=\infty$ when $n=3,2,1$, respectively, and $\partial^q_{x^q_j} B \in L^\infty_t L^\infty_x$ for $q \leq 2k$. Together with the fact that $Q^{-1}$ is an isomorphism from $L^2(\Rm^n)$ to $H^{2k}(\Rm^n) \subset W^{2k-2,\infty}(\Rm^n)$, this is enough to insure that $[Q,\Delta_B]Q^{-1} \in \calL(L^2(\Rm^n))$.
\end{proof}

\bigskip

Hypothesis (H-2) is then verified with $C(t)=[Q,\Delta_{B_t}]Q^{-1}$, the strong continuity of $C$ following from the continuity of $B$.

Finally, (H-3) follows easily from $H^{2k}(\Rm^n) \subset D(\Delta_{B_t})=H^2(\Rm^n)$, $k\geq 1$, and that $B \in \calC^0(0,T,H^{2k+2}(\Rm^n))$. We can thus apply Theorem \ref{kato} and obtain the existence of an evolution group $U$ generated by $i \Delta_{B_t}$. The fact that $U$ is an isometry on $L^2(\Rm^n)$ is a consequence of $\Re i (\Delta_{B_t} \varphi,\varphi)=0$ for every $\varphi \in H^2(\Rm^n)$.
\end{proof}


\begin{remark} When $B \in \calC^1(0,T,H^2(\Rm^n))$, a classical choice \cite{pazy} for $Q$ is $Q(t)= \lambda I-A(t)$ for $\lambda$ in the resolvent set of $A$. This allows to lower the spatial regularity of $B$ but is not verified when $B=B_t^H$. Notice that in the case when $B=B_t^H$, Proposition \ref{geneevol} can likely be improved in terms of the required spatial regularity of $B$ since we have not used the H\"older regularity in time of $B_t^H$ at all.  
\end{remark}
\subsection{Application to the magnetic Schr\"odinger equation}
We apply now the result of the preceeding section to the differential equation 
\be \label{eqgene}
\partial_t u=i \Delta_Bu+f, \qquad 0<t\leq T, \qquad u(0)=v,
\ee
where $\Delta_B=e^{i B} \circ \Delta \circ \,e^{-iB}$. As for \fref{SSE}, we say that $u \in \calC^0(0,T,H^{q}(\Rm^n)) \cap\calC^1(0,T,H^{q-2}(\Rm^n))$, $q$ non-negative integer, is a solution to \fref{eqgene} if it verifies for all $w\in \calC^1(0,T,H^{q+2}(\Rm^n))$, for all $t \in [0,T]$:
\begin{align} \label{defsol2} \nonumber
&\left( u(t), w(t)   \right)-\left( v, w(0) \right)
=\int_0^t  \left( u(s), \partial_s w(s) \right) ds
\\
& \qquad \qquad -i\int_0^t \left( u(s),\Delta_B w(s)\right) ds  +i\int_0^t \left(f,w(s) \right) ds. 
\end{align}
We have the following result:
\begin{proposition} \label{gene_exist}
Let $B \in \calC^0(0,T,H^{q+4}(\Rm^n))$, $q$ non-negative integer, and denote by $U$ the evolution operator of Proposition \ref{geneevol}. Then, for every $v \in H^{q}(\Rm^n)$ and $f \in \calC^0(0,T,H^{q}(\Rm^n))$, the function
\be \label{rep}
u(t)=U(t,0) v+\int_0^t U(t,s) f(s)ds
\ee
belongs to $\calC^0(0,T,H^{q}(\Rm^n)) \cap\calC^1(0,T,H^{q-2}(\Rm^n))$ and is the unique solution to \fref{eqgene}. Moreover, $u$ satisfies the estimate, for all $t \in [0,T]$:
\be \label{estimu}
\|u(t)\|_{H^q} \leq C \| v\|_{H^q}+C\int_0^t \| f(s) \|_{H^q}ds,
\ee
where the constant $C$ depends on $\| B\|_{\calC^0(0,T,H^{q+4}(\Rm^n))}$ when $q\neq 0$.
\end{proposition}
\begin{proof}
Consider first the case $q=2k$ with $k\geq 1$. The result then follows from \cite{katolinear}, Theorem II and the equation  \fref{eqgene} in order to obtain the regularity on $\partial_t u$. The cases $q=2k-1$, $k \geq 1$, and $q=0$  are treated by approximation: choose for instance sequences $B_\eps \in \calC^0(0,T,H^{q+9}(\Rm^n))$, $v_\eps\in H^{q+5}(\Rm^n)$ and $f_\eps \in \calC^0(0,T,H^{q+5}(\Rm^n))$ such that as $\eps \to 0$:
\bea
\label{convB1} B_\eps &\to&B \qquad \textrm{ in} \quad \calC^0(0,T,H^{q+4}(\Rm^n))\\
\label{convv} v_\eps &\to&v \qquad \textrm{ in} \quad H^{q}(\Rm^n)\\
\label{convf} f_\eps &\to&f \qquad \textrm{ in} \quad \calC^0(0,T,H^{q}(\Rm^n)).
\eea
Applying the result when $q=2k$ with $k\geq 1$, the corresponding smooth solution $u_\eps$ to \fref{eqgene} when $q=2k-1$ belongs to  $\calC^0(0,T,H^{2k+4}(\Rm^n))$ with $\partial_t u_\eps \in \calC^1(0,T,H^{2k+2}(\Rm^n))$, with the convention that $k=\frac{1}{2}$ when $q=0$. In order to pass to the limit, it is proven in \cite{katolinear}, Theorem V,  that if $i \Delta_{B_\eps}$ converges to $i \Delta_{B}$ in $\calL(H^2(\Rm^n),L^2(\Rm^n))$ a.e. $t$, and $\|\Delta_{B_\eps}\|_{\calL(H^2(\Rm^n),L^2(\Rm^n))} $ is uniformly bounded in $t$ independently of $\eps$, then 
\be \label{convU}
U_\eps(t,s) \to U(t,s) \quad \textrm{in} \quad \calL(L^2(\Rm^n)) \quad \textrm{uniformly in } (t,s),
\ee
where $U$ is the evolution operator associated to $B$. These latter conditions are direcly satisfied because of \fref{convB1}. We then write:
\bee
u_\eps(t)&=& U_\eps(t,0)v_\eps+\int_0^t U_\eps(t,s)f_\eps(s) ds, \qquad \forall t \in [0,T]\\
&=&U(t,0)v+\int_0^t U(t,s)f(s) ds+R^1_\eps+R^2_\eps=u+R^1_\eps+R^2_\eps\\
R^1_\eps&=&U_\eps(t,0)(v_\eps-v)+\int_0^t U_\eps(t,s) (f_\eps(s)-f(s))\\
R^2_\eps&=&(U_\eps(t,0)-U(t,0))v+\int_0^t (U_\eps(t,s)-U(t,s)) f(s)) ds.
\eee
Using \fref{convU} and the strong convergence of $v_\eps$ and $f_\eps$, we then obtain that $u_\eps \to u$ in $\calC^0(0,T,L^2(\Rm^n))$. Assume first that $q \neq 0$. In order to get the announced better regularity on $u$, we use the fact that $u_\eps \in \calC^0(0,T,\calC^{2k+2}(\Rm^n))$ and $\partial_t u_\eps \in \calC^1(0,T,\calC^{2k}(\Rm^n))$ thanks to standard Sobolev embeddings for $n \leq 3$. We can then differentiate equation \fref{geneevol}, and find using the representation formula
\be \label{eqDbeta}
 D^\beta u_\eps(t)= U_\eps(t,0) D^\beta v_\eps+\int_0^t U_\eps(t,s)(D^\beta f_\eps(s)+[D^\beta,L_\eps(s)]u_\eps(s)) ds,
\ee
where $1\leq |\beta| \leq q$ and 
$$D^\beta:=\frac{\partial^{\beta_1}}{\partial x_1^{\beta_1}} \times \cdots \times \frac{\partial^{\beta_n}}{\partial x_n^{\beta_n}}, \qquad \beta=(\beta_1,\cdots,\beta_n), \quad |\beta|=\beta_1 + \cdots + \beta_n,$$
 and  $L_\eps(s)$ is defined in \fref{defL} with $B$ replaced by $B_\eps$. Only the term involving the commutator requires some attention. Using \fref{convB1}, we can show that for all $s \in [0,T]$,
$$
\| [D^\beta,L_\eps(s)]u_\eps(s)\|_{L^2} \leq C \|u_\eps(s)\|_{H^{|\beta|}},
$$
where the constant $C$ is independent of $\eps$. Together with \fref{convB1}-\fref{convv}-\fref{convf}-\fref{convU}-\fref{eqDbeta} and the Gronwall lemma, this yields a uniform bound for $u_\eps$ in $\calC^0(0,T,H^q(\Rm^n))$. Using this latter bound along with \fref{convB1}-\fref{convv}-\fref{convf}-\fref{convU}-\fref{eqDbeta} and equation \fref{eqgene} for the smooth solution $u_\eps$ in order to estimate $\partial_t u_\eps$, it is then not difficult to show that $(u_\eps)_\eps$ is a Cauchy sequence in $\calC^0(0,T,H^{q}(\Rm^n)) \cap\calC^1(0,T,H^{q-2}(\Rm^n))$, whose limit $u$ satisfies estimate \fref{estimu} and  \fref{defsol2}. When $q=0$, it suffices to use equation \fref{eqgene} for the smooth solution $u_\eps$ in order to show that $(\partial_t u_\eps)_\eps$ is Cauchy in $\calC^0(0,T,H^{-2}(\Rm^n))$. This proves the existence, the representation formula \fref{rep} and estimate \fref{estimu}.

Uniqueness is straightforward in the case $q \geq 1$ since solutions to \fref{eqgene} are regular enough to be used as test functions and to obtain after an integration by part that $\Im (\nabla (e^{-i B_t} u),\nabla (e^{-i B_t} u ))=0$. When $q=0$, we use the adjoint formulation of \fref{eqgene}. The difference between two solutions to \fref{eqgene} satisfies in the case of a test function $w \in \calC^0(0,T,H^2(\Rm^n)) \cap \calC^1(0,T,L^2(\Rm^n))$,
\begin{align*}
&\left( u(t), w(t)\right)=\int_0^t \left ( u(s), \partial_s w(s)+(i\Delta_{B_s})^* w(s) \right) ds,
\end{align*}
where $(i\Delta_{B_s})^*=-i\Delta_{B_s}$ is the adjoint of $i \Delta_{B_s}$. Let $t \in [0,T]$, pick some $w_0 \in H^2(\Rm^n)$ and let $w(s)=z(t-s)$ where $z(s)$ is the solution to $\partial_s z(s)=(i \Delta_{B_{t-s}})^* z(s)$, $z(0)=w_0$, $0<s<t$. Adapting Proposition \ref{geneevol} to the operator $(i\Delta_{B_s})^*$ , Theorem \ref{kato} yields that $z \in \calC^0(0,T,H^2(\Rm^n)) \cap \calC^1(0,T,L^2(\Rm^n))$. Hence, $\partial_s w(s)+(i\Delta_{B_s})^* w(s)=0$ $x$ a.e., $w(t)=w_0$ and it comes, for all $t \in [0,T]$:
$$
\left( u(t), w_0 \right)=0, \qquad \forall w_0 \in H^2(\Rm^n),
$$
so that $u=0$. This ends the proof.
\end{proof}
\bigskip

We use the result of the last Proposition to prove that the non-linear magnetic Schr\"odinger equation 
\be \label{mag}
\partial_t \varphi=i \Delta_{B^H_t} \varphi +g(\varphi), \qquad 0<t\leq T, \qquad u(0)=\Psi_0,
\ee
admits a unique solution $\P$ almost surely in the same sense as \fref{defsol2}:
\begin{theorem} \label{th_mag}
Assume that \textbf{H} is satisfied. Suppose moreover that \fref{assumQ} is verified for $V=H^{q+4}(\Rm^n)$, $q$ non-negative integer. Then, for every $\Psi_0 \in H^q(\Rm^n)$, there exists a maximal existence time $T_M>0$ and a unique function $\varphi \in \calC^0(0,T_M,H^{q}(\Rm^n)) \cap\calC^1(0,T_M,H^{q-2}(\Rm^n))$ verifying \fref{mag} for $t\in[0,T_M]$ which admits the following representation formula:
$$
\varphi(t)= U(t,0) \Psi_0+ \int_0^t U(t,s) g(\varphi(s))ds,
$$
where $U=\{U(t,s)\}$ is the evolution operator generated by the operator
$$
i\Delta_{B_t^H} =i e^{i B_t^H}\circ \Delta \circ e^{-i B_t^H}.
$$
If moreover $\Im g(\varphi) \overline{\varphi}=0$, then for all $t\in[0,T_M]$
\be \label{charge}
\| \varphi(t)\|_{L^2}=\| \varphi(0)\|_{L^2}.
\ee
If $g$ is globally Lipschitz on $H^q(\Rm^n)$, then the solution exists for all time $T<\infty$.
\end{theorem}
\begin{proof} 
The proof is very classical and relies on Proposition \ref{gene_exist} and a standard fixed point procedure. First of all, \fref{assumQ} insures that $\P$ almost surely, $B^H \in \calC^0(0,T,H^{q+4}(\Rm^n))$, which allows us to define an evolution operator $U$ according to Proposition \ref{geneevol}. The rest of the proof follows the usual arguments of for instance \cite{pazy}, Theorem 1.4, Chapter 6, that we sketch here for completeness. Given $\Psi_0$ in $H^{q}(\Rm^n)$ and $\varphi \in \calC^0(0,t_1,H^q(\Rm^n))$ for some $t_1>0$ to be fixed later on, denote by $u:=F(\varphi)$ the solution to $\fref{eqgene}$ in $\calC^0(0,t_1,H^{q}(\Rm^n)) \cap\calC^1(0,t_1,H^{q-2}(\Rm^n))$ where $f=g(\varphi)$ belongs to $\calC^0(0,t_1,H^{q}(\Rm^n))$ thanks to hypothesis \textbf{H}. Using the latter, estimate \fref{estimu}, following the aforementioned theorem of \cite{pazy}, one can establish the existence of  $M>0$ and a time $t_1(M)$ such that $F$ maps the ball of radius $M$ of $\calC^0(0,t_1,H^{q}(\Rm^n))$ centered at 0 into itself. In this ball, the function $g$ being uniformly Lipschitz on $H^p(\Rm^n)$, $0\leq p\leq q$ according to hypothesis \textbf{H}, existence and uniqueness of a fixed point  of $F$ in $\calC^0(0,t_1,H^{q}(\Rm^n))$, denoted by $\varphi^\star$, follows from the contraction principle. Moreover, this solution verifies the representation formula \fref{rep} with $f=g(\varphi^\star) \in \calC^0(0,t_1,H^{q}(\Rm^n))$ and according to Proposition \ref{gene_exist}, belongs in addition to $\calC^1(0,t_1,H^{q+2}(\Rm^n))$ and satisfies \fref{mag}. The existence of a maximal time of existence $T_M$  is established following the same lines as \cite{pazy}. When $g$ is globally Lipschitz in $H^q(\Rm^n)$, then $T_M<\infty$ by the Gronwall Lemma.

Regarding the conservation of charge \fref{charge}, the case $q \geq 1$ is direct since the solution $\varphi^\star$ is regular enough to be used as a test function in \fref{defsol2} (after interpretation of $(\cdot, \cdot )$ as the $H^{-1}-H^1$ duality pairing when $q=1$) and it then suffices to take the imaginary part of the equation. When $q=0$, we use a regularization procedure very similar to that of the proof of Proposition \ref{gene_exist}, the details are left to the reader.
 \end{proof}

\section{Back to the stochastic Schr\"odinger equation} \label{back}
We apply the result of the last section to prove Theorem \ref{th1}. Owing the solution $\varphi$ of Theorem \ref{th_mag}, it suffices to show (i) that $e^{-i B_t^H} \varphi$ is a solution to \fref{defsol}, which will follow from the regularity of $\varphi$ and Lemma \ref{chain2}, this yields existence; and (ii) that all solutions to \fref{defsol} with the corresponding regularity read $e^{-i B_t^H} u$ where $u$ is a solution to \fref{mag}, which yields uniqueness since \fref{mag} has a unique solution. As explained earlier, the last step requires a regularization procedure since test functions of the form $w=e^{-i B_t^H} z$, with $z$ smooth, are not differentiable in time and cannot be used directly in \fref{defsol}. In the whole proof, $T$ denotes some time $T\geq T_M$.\\

\noindent \textbf{Proof of Theorem \ref{th1}.} \textit{Existence.} Let $\varphi$ be the unique solution to \fref{mag} according to Theorem \ref{th_mag} and define, for any test function $w \in \calC^1(0,T_M,H^{q+2}(\Rm^n))$,
$$
F(B^H_t,t)=( e^{-i B_t^H} \varphi(t), w(t)). 
$$
We verify that $F$ satisfies the hypotheses of Lemma \ref{chain2}. First of all, $F$ is clearly continuously differentiable w.r.t. the first variable and for all $v \in H^{q+4}(\Rm^n)$, let $\phi(t):=\partial_1 F(B^H_t,t)(v)=i ( e^{-i B_t^H} \varphi(t) v, w(t))$. Second of all, we need to show that $\phi \in \calC^{0,\lambda}(0,T)$ for $\lambda$ verifying $\lambda+\gamma>1$, together with the bound \fref{hypphi}. To this goal, we have for $(t,s) \in [0,T_M]^2$:
\bee
\phi(t)-\phi(s)&=&i\left( (e^{-i B_t^H}-e^{-i B_s^H}) \varphi(t) v, w(t)\right)+i\left(e^{-i B_s^H} (\varphi(t)-\varphi(s)) v, w(t)\right)\\
&&+i\left(e^{-i B_s^H} \varphi(s) v, (w(t)-w(s))\right)\\
&:=&T_1+T_2+T_3.
\eee
We treat each term separately. We have, using standard Sobolev embeddings for $n\leq 3$:
\bea \nonumber
|T_1| &\leq& C \|\varphi(t)\|_{L^2} \|w(t)\|_{L^\infty}  \|v\|_{L^\infty} \|B_t^H-B_s^H\|_{L^2}\\ \nonumber
&\leq & C (t-s)^{\gamma} \|v\|_{H^{q+4}} \|B_t^H\|_{\calC^{0,\gamma}(0,T_M,L^2)}\\
&\leq & C (t-s)^{\gamma} \|v\|_{H^{q+4}}, \label{estimT1}
\eea
for all $0\leq \gamma<H$. Regarding the term $T_2$, notice that the product $e^{i B_t^H} \overline{v} w$ belongs to $H^{q+2}(\Rm^n)$ when $n \leq 3$, so that since $\partial_t \varphi \in \calC^{0}(0,T_M,H^{q-2}(\Rm^n))$, we can write
$$
\left(e^{-i B_s^H} (\varphi(t)-\varphi(s)) v, w(t)\right)=\int_s^t \langle \partial_\tau \varphi(\tau) , e^{ i B_s^H} \overline{v} w(t) \rangle_{H^{q-2},H^{q+2}} d\tau,
$$ 
where when $q\geq 2$, the pairing $\langle \cdot, \cdot \rangle_{H^{q-2},H^{q+2}}$ is replaced by the $L^2$ inner product. Hence,
$$
|T_2| \leq |t-s| \|w(t)\|_{H^{q+2}} \|v\|_{H^{q+4}} \|B_t^H\|_{\calC^{0}(0,T_M,H^{q+4})} \|\partial_t \varphi\|_{\calC^{0}(0,T_M,H^{q-2})}  \leq C |t-s| \|v\|_{H^{q+4}}.
$$
Estimation of $T_3$ is straightforward and leads to a similar estimate as above. This, together  with \fref{estimT1} yields that
$$
\| \phi\|_{\calC^{0,\gamma}(0,T_M)} \leq C \|v\|_{H^{q+4}}.
$$
Since $\frac{1}{2}<H$, we can pick $H=\frac{1}{2}+\eps$ and $\gamma=\frac{1}{2}+\frac{\eps}{2}$ such that $2\gamma>1$ and the assumption on $\phi$ of Lemma \ref{chain2} is verified. It remains to show that $\partial_2 F$ exists and is continuous, and this is a consequence of the fact that $\partial_t \varphi \in \calC^{0}(0,T_M,H^{q-2}(\Rm^n))$.
Applying Lemma \ref{chain2} then yields
\begin{align} \label{eqpphi}
&( e^{-i B_t^H} \varphi(t), w(t))-( \Psi_0, w(0))=\int_0^t\langle  \partial_\tau \varphi(\tau), e^{i B_\tau^H} w(\tau)\rangle_{H^{q-2},H^{q+2}} d\tau\\ \nonumber
&+\int_0^t( e^{-i B_\tau^H} \varphi(\tau), \partial_\tau w(\tau)) d\tau+i
\sum_{p \in \Nm} \lambda_p \int_0^t  ( e^{-i B_\tau^H} \varphi(\tau)e_p, w(\tau)) d \beta^H_p(\tau).
\end{align}
In order to conclude, picking $w(t,x)=w(x) \in H^{q+2}(\Rm^n)$ in \fref{defsol2} with $f=g(\varphi)$, it comes that \fref{mag} is verified in $H^{q-2}(\Rm^n)$ for all $t\in [0,T_M]$ and almost surely. This yields $\partial_\tau \varphi= i \Delta_{B_t^H} \varphi-i g(\varphi)$ in $H^{q-2}(\Rm^n)$, and replacing  $\partial_\tau \varphi$ by its latter  expression in \fref{eqpphi}, setting
$\Psi=e^{-i B_t^H} \varphi \in \calC^0(0,T_M,H^q(\Rm^n)) \cap \calC^{0,\gamma}(0,T_M,H^{q-2}(\Rm^n))$ for all $0\leq \gamma<H$, finally yields \fref{defsol}.

\textit{Uniqueness, Step 1: regularization.} Starting from a solution $\Psi$ to \fref{defsol} with the above regularity, we would like to choose the test function $w=e^{-i B_\tau^H} z$ for some regular function $z$ in order to recover the weak formulation of \fref{mag}, which admits a unique solution. This is not allowed of course since $B^H$ is not differentiable. The solution is to use the H\"older regularity of $\Psi$ in order to reinterpret the term
$$
\int_0^t  \left( \Psi(s), \partial_s w(s) \right) ds
$$
as a fractional integral. To this end, let 
$$
B^{H,\eps}(t,x):=\sum_{p \in \Nm} \lambda_p e_p(x) \beta_p^{H,\eps}(t) 
$$
where $\beta^{H,\eps}_p$ is a $\calC^1$ regularization of $\beta^{H}_p$ such that $\beta^{H,\eps}_p \to \beta^{H}_p$ in $\calC^{0,\gamma}(0,T)$ almost surely for all $p$ and $\|\beta^{H,\eps}_p\|_{\calC^{0,\gamma}(0,T)} \leq \|\beta^{H}_p\|_{\calC^{0,\gamma}(0,T)}$, $0\leq \gamma <H$. We have 
\be \label{conVV}
B^{H,\eps}\to B^{H}\quad \textrm{in}\quad \calC^{0,\gamma}(0,T,H^{q+4}(\Rm^n)),\qquad \P \quad \textrm{almost surely}.
\ee 
Indeed:
$$
\| B^{H,\eps}-B^{H}\|_{\calC^{0,\gamma}(0,T,H^{q+4})} \leq \sum_{p \in \Nm} \lambda_p \| e_p \|_{H^{q+4}} \|\beta^{H,\eps}_p -\beta^{H}_p\|_{\calC^{0,\gamma}(0,T)}
$$
and
$$
\|\beta^{H,\eps}_p -\beta^{H}_p\|_{\calC^{0,\gamma}(0,T)} \leq 2 \|\beta^{H}_p\|_{\calC^{0,\gamma}(0,T)},
$$
which, together with the convergence of $\beta_p^{H,\eps}$ to $\beta_p^{H}$ in $\calC^{0,\gamma}$, \fref{defK} and the Weierstrass rule gives the desired result. Set $w_\eps=e^{-i B_t^{H,\eps}} z$ where $z \in \calC^1(0,T,H^{q+2}(\Rm^n))$. Then
\begin{align*}
&\int_0^t  \left( \Psi(s), \partial_s w_\eps(s) \right) ds\\
&\qquad =\int_0^t  \left( \Psi(s)e^{i B_s^{H,\eps}}, \partial_s z(s) \right)ds-i\sum_{p\in \Nm} \lambda_p \int_0^t  \left( \Psi(s)e^{i B_s^{H,\eps}},z(s) e_p\right)  (\beta_p^{H,\eps}(s))' ds,
\end{align*}
where all permutations of sum and integrals were permitted since the series defining $B^{H,\eps}$ is normally convergent in $\calC^1(0,T,H^{q+4}(\Rm^n))$. We now use the fact that for a continuously differentiable function $f$, $D_{t-}^\alpha f \to f'$ as $\alpha \to 1$, see \cite{zahle} section 1, and that $D_{t-}^{\alpha+\beta} f=D_{t-}^{\alpha} D_{t-}^{\beta} f$, where the operator $D_{t-}^{\alpha}$ is defined in section \ref{prelim}. We then introduce, for $\frac{1}{2}<1-\mu<\gamma$,
\bee
I^\eps_p&:=&\int_0^t  \left( \Psi(s)e^{i B_s^{H,\eps}},e_p\right)  (\beta_p^{H,\eps}(s)-\beta_p^{H,\eps}(t))' ds\\
&=&\lim_{\alpha \to 1} I^{\eps,\alpha}_p := \lim_{\alpha \to 1} \int_0^t  \left( \Psi(s)e^{i B_s^{H,\eps}},e_p z(s)\right) D_{t-}^{\alpha-\mu} D_{t-}^{\mu} (\beta_p^{H,\eps})_{t-}(s)ds,
\eee
where $(\beta_p^{H,\eps})_{t-}(s)=\beta_p^{H,\eps}(s)-\beta_p^{H,\eps}(t^-)$. Owing the fact that $( \Psi(s)e^{i B_s^{H,\eps}},z(s) e_p)\in \calC^{0,\gamma}(0,T_M)$ for any $0\leq \gamma<H$ and using the fractional integration by part formula of \cite{zahle} section 1, we find
$$
I^{\eps,\alpha}_p=(-1)^{\alpha-\mu} \int_0^t  D_{0+}^{\alpha-\mu} \left( \Psi(s)e^{i B_s^{H,\eps}},z(s) e_p\right)  D_{t-}^{\mu} (\beta_p^{H,\eps})_{t-}(s)ds.
$$
Moreover, we can send $\alpha$ to one above thanks to dominated convergence since the term $( \Psi(s)e^{i B_s^{H,\eps}},z(s) e_p)$ belongs to $\calC^{0,\gamma}(0,T_M)$ and $1-\mu<\gamma$ in order to obtain
$$
I^\eps_p=(-1)^{1-\mu} \int_0^t  D_{0+}^{1-\mu} \left( \Psi(s)e^{i B_s^{H,\eps}},z(s) e_p\right)  D_{t-}^{\mu} (\beta_p^{H,\eps})_{t-}(s)ds.
$$
We derive below some estimates needed to pass to the limit $\eps \to 0$.\\
\noindent \textit{Uniqueness, Step 2: uniform estimates.} Recall that $w_\eps=e^{-i B_t^{H,\eps}} z$ and define first, with $w=e^{-i B_t^{H}} z$:
$$\phi^\eps(s)=( \Psi(s),(w_\eps(s)-w(s)) e_p):=( \Psi(s),r_\eps(s)e_p).$$
In order to estimate $D_{0+}^{1-\mu} \phi^\eps$, we write
\bee
\phi^\eps(t)-\phi^\eps(s)&=&( \Psi(t)- \Psi(s),r_\eps(t)e_p)+( \Psi(s),(r_\eps(t)-r_\eps(s))e_p):=T_1+T_2.
\eee
For the term $T_1,$ we use the $\calC^{0,\gamma}$ regularity of $\Psi$ in $H^{q-2}$, while we use that of $r^\eps$ for $T_2$. It comes with the help of standard Sobolev embeddings:
\bee
|T_1| &\leq& |t-s|^\gamma \|\Psi\|_{\calC^{0,\gamma}(0,T_M,H^{q-2})}\|r_\eps(t)\|_{H^{q+2}} \|e_p\|_{H^{q+2}}\\
|T_2| &\leq&  |t-s|^\gamma \|\Psi(t)\|_{L^2} \|e_p\|_{L^\infty} \| r_\eps \|_{\calC^{0,\gamma}(0,T,L^{2})}.
\eee
Since $1-\mu<\gamma$, this gives
\be \label{estphiW}
\| \phi^\eps\|_{W_{1-\mu,1}(0,T_M)} \leq C \|e_p\|_{H^{q+4}} \| B^{H,\eps}-B^H \|_{\calC^{0,\gamma}(0,T,H^{q+4}(\Rm^n))}.
\ee
On the other hand, using the notation of section \ref{prelim}, we find
\be \label{estimbet}
|D_{t-}^{\mu} (\beta_p^{H,\eps})_{t-}(s)| \leq \Lambda_{1-\mu}(\beta^{H,\eps}_p) \leq C \|\beta_p^{H,\eps}\|_{\calC^{0,\gamma}(0,T)} \leq C \|\beta_p^{H}\|_{\calC^{\gamma}(0,T)} < \infty, \ee
and 
\be \label{estimbet2}
|D_{t-}^{\mu} (\beta_p^{H,\eps}-\beta_p^{H})_{t-}(s)| \leq \Lambda_{1-\mu}(\beta^{H,\eps}_p-\beta_p^{H}) \leq C \|\beta_p^{H,\eps}-\beta_p^{H}\|_{\calC^{0,\gamma}(0,T)}. \ee


\noindent \textit{Uniqueness, Step 3: passing to the limit.} We have all needed now to pass to the limit in the weak formulation \fref{defsol}. Plugging $w^\eps(t)=e^{-i B_t^{H,\eps}} z \in \calC^1(0,T,H^{q+2}(\Rm^n))$ yields
\begin{align} \label{defsol3} \nonumber
&\left( \Psi(t), w^\eps(t)   \right)-\left( \Psi_0, z(0) \right)
=\int_0^t  \left( \Psi(s), \partial_s w^\eps(s) \right) ds
\\[3mm]
& -i\int_0^t \left( \Psi(s),\Delta w^\eps(s)\right) ds  +i  \int_0^t\left(\Psi(s),w^\eps(s)  d B_s^H\right)+i\int_0^t \left(g(\Psi(s)),w^\eps(s) \right) ds.
\end{align}
We have
$$
\int_0^t  \left( \Psi(s), \partial_s w^\eps(s) \right) ds=\int_0^t  \left( \Psi(s)e^{i B_t^{H,\eps}}, \partial_s z(s) \right) ds-i\sum_{p \in \Nm^*} I_p^\eps.
$$
Using \fref{estimint}-\fref{conVV}-\fref{estphiW}-\fref{estimbet}-\fref{estimbet2} as well as \fref{defK}, we can pass to the limit in the latter equation and obtain that, $\forall t \in [0,T_M]$:
$$
\lim_{\eps \to 0} \int_0^t  \left( \Psi(s), \partial_s w^\eps(s) \right) ds=\int_0^t  \left( \Psi(s)e^{i B_s^{H}}, \partial_s z(s) \right) ds-i  \int_0^t\left(\Psi(s)e^{i B_s^{H}}, z(s) d B_s^H\right).
$$
Similar arguments can be employed to pass to the limit in the remaining terms of \fref{defsol3}. The stochastic integrals simplify and we are left with
\begin{align} \label{defsol4} \nonumber
&\left( \Psi(t)e^{i B_t^{H}}, z(t)   \right)-\left( \Psi_0, z(0) \right)
=\int_0^t  \left( \Psi(s)e^{i B_s^{H}}, \partial_s z(s) \right) ds
\\[3mm]
& -i\int_0^t \left( \Psi(s),\Delta (e^{-i B_s^{H}}z(s))\right) ds  +i\int_0^t \left(g(\Psi(s)),e^{-i B_s^{H}}z(s) \right) ds.
\end{align}
Hence, $\Psi e^{i B_t^{H}}$ verifies the magnetic Schr\"odinger equation $\fref{defsol2}$ with $f=g(\Psi(s))e^{i B_s^{H}}=g(\Psi(s)e^{i B_s^{H}})$. Since the latter admits a unique solution $\varphi \in \calC^0(0,T_M,H^q(\Rm^n)) \cap \calC^1(0,T_M,H^{q-2}(\Rm^n))$ according to Theorem \ref{th_mag}, we can conclude that \fref{SSE} admits a unique solution. The representation formula \fref{repre_th} follows then without difficulty with the identification $\Psi e^{i B_t^{H}}=\varphi$ and Theorem \ref{th_mag}. This ends the proof of Theorem \ref{th1}.
\section{Appendix}
\subsection{Proof of Lemma \ref{chain2}} First of all, we know by section \ref{prelim} that $B_t^{H}$ belongs to $E:=\calC^{0,\gamma}(0,T,V)$, $\P$ almost surely for $0\leq \gamma<H$. We proceed by approximation in order to apply the change of variables formula \fref{chain} valid in finite dimensions. Let $$B_t^{H,N}(x):=\sum_{p=0}^N \lambda_p e_p(x) \beta^H_p(t),$$ so that, 
\be \label{convB}B_t^{H,N} \to B_t^{H} \quad \textrm{in   } E, \qquad \P \textrm{   almost surely,}
\ee
 thanks to \fref{assumQ} and \fref{defK}.  We have moreover the bound $\|B_t^{H,N}\|_E \leq \|B_t^{H}\|_E:=M$.  Since $F$ is $\calC^1$, and $\phi \in \calC^{0,\lambda}(0,T)$ with $\lambda+\gamma>1$, we can use \fref{chain} and find, for $0\leq s \leq t \leq T$ fixed:
\bee
F(B_t^{H,N},t)-F(B_s^{H,N},s)=\int_s^t \partial_2 F(B_{\tau}^{H,N},\tau) d\tau+
\sum_{p=0}^N \lambda_n \int_s^t \partial_1 F(B_{\tau}^{H,N},\tau)(e_p)\, d \beta_p^H(\tau).
\eee
By continuity of $F$, it is direct to pass to the limit in the left hand side. The same holds for the first term of the right hand side thanks to dominated convergence and the fact that $\partial_2 F$ is continuous and $B_\tau^{H,N}$ is bounded in $E$ independently of $N$. Regarding the last term, let $\phi^N_p(\tau)=\partial_1 F(B_{\tau}^{H,N},\tau)(e_p)$ and 
$$
f_p^N:=\int_s^t \phi_p^N(\tau) \, d \beta_p^H(\tau)=(-1)^\alpha \int_s^t D_{s+}^\alpha \phi_p^N(\tau) D_{t-}^{1-\alpha} (\beta_p^H)_{t-}(s) d\tau,
$$
by \fref{stiel} for some $\alpha$ verifying $\alpha<\lambda$ and $1-\alpha<\gamma$. Since $\|B_t^{H,N}\|_E \leq M$, we have by \fref{hypphi}
$$
|\tau-s|^{-\alpha-1}|(\phi_p^N(\tau)-\phi_p^N(s))| \leq |\tau-s|^{\lambda-\alpha-1} \|\phi_p^N \|_{\calC^{0,\lambda}(0,T)} \leq C_M |\tau-s|^{\lambda-\alpha-1} \|e_p \|_V,
$$
where $\lambda-\alpha>0$. The latter estimate, dominated convergence, \fref{convB} together with the continuity of $\partial_1 F$ yield first that $D_{s+}^\alpha \phi_p^N(\tau) \to D_{s+}^\alpha \phi_p(\tau)$ a.e. where $\phi_p(\tau)=\partial_1 F(B_{\tau}^{H},\tau)(e_p)$. Then, since $|D_{t-}^{1-\alpha} (\beta_p^H)_{t-}(s)| \leq \Lambda_\alpha(\beta^H_p)<\infty$, $\P$ almost surely, we have 
\be \label{estimDphi}
|D_{s+}^\alpha \phi_p^N(\tau) D_{t-}^{1-\alpha} (\beta_p^H)_{t-}(s) | \leq C |\tau-s|^{\lambda-\alpha-1}\|e_p \|_V
\ee
so that dominated convergence implies that $f_p^N \to f_p=\int_s^t \phi_p(\tau) \, d \beta_p^H(\tau)$, for all $p \in \Nm$. Finally, since
$$
\lambda_p |f_p^N| \leq C \lambda_p \|e_p \|_V,
$$
thanks to \fref{estimDphi}, and moreover \fref{assumQ} holds, we can apply the Weierstrass rule and conclude that $\P$ almost surely:
$$
\lim_{N\to \infty} \sum_{p=0}^N \lambda_n f_p^N=\sum_{p=0}^\infty \lambda_n \int_s^t \partial_1 F(B_{\tau}^{H},\tau)(e_p)\, d \beta_p^H(\tau).
$$
This ends the proof.
\subsection{Proof of Lemma \ref{fubini}}
The hypothesis on $F$ show that the integral on the left is well-defined and that
$$
\sum_{p \in \Nm} \lambda_p \int_0^t \left(\int_{\Rm^n}  F_{\tau,x}(e_p) dx \right) d \beta^H_p(\tau) = \lim_{N\to \infty } \sum_{p=0}^N \lambda_p \int_0^t \left(\int_{\Rm^n}  F_{\tau,x}(e_p) dx \right) d \beta^H_p(\tau):=\lim_{N\to \infty } I_N.
$$
Morever, for $1-H<\alpha<\frac{1}{2}$, we have that $|D_{t-}^{1-\alpha} (\beta_p^H)_{t-}(s)| \leq \Lambda_\alpha(\beta^H_p)<\infty$, $\P$ almost surely and $D_{0+}^\alpha F_{t,x}(e_n) \in L^1((0,T)\times \Rm^n)$ since $F_{t,x}(e_p)\in W_{\alpha,1}(0,T,L^1(\Rm^n))$. Hence, using the definition of the stochastic integral and Fubini Theorem, it comes
\bee
I_N&=&(-1)^{|\alpha|}\sum_{p=0}^N \lambda_p \int_0^t \left[D_{0+}^\alpha \int_{\Rm^n} F_{\tau,x}(e_p)(\tau) dx \right]\left[D_{t-}^{1-\alpha} (\beta_p^H)_{t-}(\tau)\right] d\tau
\\
&=&
(-1)^{|\alpha|}\sum_{p=0}^N \lambda_p \int_{\Rm^n} \left(\int_0^t \left[D_{0+}^\alpha F_{\tau,x}(e_p)(\tau)  \right] \left[D_{t-}^{1-\alpha} (\beta_p^H)_{t-}(\tau)\right] d\tau \right)dx\\
&=& \int_{\Rm^n} \left(\sum_{p=0}^N \lambda_p \int_0^t  F_{t,x}(e_p)  d \beta^H_p(\tau) \right) dx:=\int_{\Rm^n} f_N(x) dx.
\eee
Moreover, $\P$ almost surely:
\bee
\|f_N\|_{L^1} &\leq&  \sum_{p=0}^N \lambda_p  \left\| \|F_{t,x}(e_p)\|_{W_{\alpha,1}(0,T)} \right\|_{L^1}  \Lambda_\alpha(\beta_p^H)\\
& \leq& C \|F\|_{W_{\alpha,1}(0,T,\calL(V,L^1))} \sum_{p=0}^N \lambda_p \|e_p\|_V \Lambda_\alpha(\beta_p^H),
\eee
so that thanks to \fref{finitelamb}, the series defining $f_N$ converges strongly in $L^1(\Rm^n)$ and almost surely. This  yields
$$
\lim_{N\to \infty } I_N=\int_{\Rm^n} \left(\sum_{p=0}^\infty \lambda_p \int_s^t  F_{t,x}(e_n)  d \beta^H_p(\tau) \right) dx
$$
and ends the proof.
\subsection{Proof of Remark \ref{rem2}}
For $q \geq 2$, using the regularity $\Psi \in \calC^0(0,T_M,H^q(\Rm^n)) \cap \calC^{0,\gamma}(0,T_M,H^{q-2}(\Rm^n))$ and picking $w \in L^2(\Rm^n)$, we can recast \fref{defsol} as
\begin{align*} 
&\left( \Psi(t), w   \right)-\left( \Psi_0, w \right)=
 -i\int_0^t \left( \Delta \Psi(s), w\right) ds  +i  \int_0^t\left(\Psi(s), w  d B_s^H\right)+i\int_0^t \left(g(\Psi(s)),w \right) ds. 
\end{align*}
Since the mapping $F: e_p \to \overline{\Psi(s)} w e_p$ belongs to $\calC^{0,\gamma}(0,T_M,\calL(V,L^1(\Rm^n)))$, we can use Lemma \ref{fubini} together with Fubini theorem to arrive at 
\begin{align*} 
&\left( \Psi(t), w   \right)-\left( \Psi_0, w \right)=
\\[3mm]
&-i \left( \int_0^t \Delta \Psi(s) ds, w\right)  +i  \left( \int_0^t\Psi(s)d B_s^H, w  \right)+i\left( \int_0^t g(\Psi(s))ds,w \right), 
\end{align*}
which yields the desired result.
\footnotesize{
\bibliography{bibliography} \bibliographystyle{siam}}
\end{document}